\documentclass{amsart}

\DeclareMathSymbol{\twoheadrightarrow}  {\mathrel}{AMSa}{"10}

\def\H{{\mathbb H}}

\def\Z{{\mathbb Z}}

\def\P{{\mathbb P}}

\def\f{{\tilde F}}
                     \def\f0{{\mathfrak f}}

             \def\K{\mathrm{K}}

\def\A8{{\mathbf A}_8}
\def\Bir{\mathrm{Bir}}

\def\div{\mathrm{div}}
\def\Hom{\mathrm{Hom}}

\def\M{\mathrm{M}}
\def\A{\mathbf{A}}

\def\dim{\mathrm{dim}}

       \def\G{{\mathfrak G}}
        \def\K{{\mathbf K}}
         \def\H{{\mathbf H}}
          \def\l1{{\mathbf 1}}

                                                   \def\OC{{\mathcal O}}
                                                   \def\M{{\mathcal M}}
                         \def\AA{\mathbb{A}}

                         \def\A{\mathbb{A}}
                                                    \def\div{\mathrm{div}}

\newtheorem{thm}{Theorem}[section]

\newtheorem{cor}[thm]{Corollary}

\theoremstyle{definition}
\newtheorem{defn}[thm]{Definition}

\newtheorem{rem}[thm]{Remark}

\title[Theta groups]{Theta groups and products of  abelian and rational varieties}

\author[Yuri G.\ Zarhin]{Yuri G.\ Zarhin}
\address{Department of Mathematics, Pennsylvania State University,
University Park, PA 16802, USA}
\address{Institute for Mathematical Problems in Biology, Russian Academy of
Sciences, Pushchino, Moscow Region, Russia}\email{zarhin\char`\@math.psu.edu}

\begin{document}

\begin{abstract}
We prove that an analogue of Jordan's theorem on finite subgroups of general linear groups does not hold for the groups of birational automorphisms of products of an elliptic curve and the projective line. This gives a negative answer to a question of V. L. Popov.
\end{abstract}


\maketitle

\section{Introduction}

Throughout this paper, $k$ is an algebraically closed field of characteristic zero, $\A^1$ and $\P^1$ the affine line and projective line respectively (both over $k$).
If $U$ is an irreducible algebraic variety over $k$ then we write $k[U], k(U)$ and $\Bir(U)$ for its ring ($k$-algebra) of regular functions, the field of rational functions and the group of birational $k$-automorphisms respectively.

The following definition was inspired by the classical theorem of Jordan about finite subgroups of general linear groups.

\begin{defn}[Definition 2.1 of \cite{Popov}]
A group $B$ is called a {\sl Jordan group} if there exists a positive integer $J_B$ such that every finite subgroup $B_1$ of $B$ contains a (normal) commutative subgroup, whose index in $B_1$ is at most $J_B$.
\end{defn}

 V. L. Popov \cite[Sect. 2]{Popov} posed a question whether $\Bir(Y)$ is a Jordan group when $Y$ is an irreducible surface. He obtained a positive answer to his question for almost all smooth projective minimal surfaces. One of the few remaining cases is a product $E\times\P^1$ of an elliptic curve $E$ and the projective line.

 Our main result is the following statement, which gives a negative answer to Popov's question.

 \begin{thm}
 \label{elliptic}
 If $E$ is an elliptic curve over $k$ then $\Bir(E\times\P^1)$  is not a Jordan group.
 \end{thm}

 Since, $U\times \A^1$ is birationally isomorphic to $U\times \P^1$, the groups
$\Bir(U\times \A^1)$ and $\Bir(U\times \P^1)$ are isomorphic and
 Theorem \ref{elliptic} becomes equivalent to the assertion that $\Bir(E\times\A^1)$  is not a Jordan group, which, in turn, is a special case of the following statement.

\begin{thm}
\label{main} Let $X$ be an abelian variety of positive dimension  over $k$. Then  $\Bir(X\times\A^1)$  is not a Jordan group.
\end{thm}

\begin{cor}
\label{rational}

 Let $X$ be an abelian variety of positive dimension  over $k$ and $Z$ is a rational variety of positive dimension  over $k$. Then $\Bir(X\times Z)$  is not a Jordan group.
\end{cor}

\begin{proof}[Proof of Corollary \ref{rational} (modulo Theorem \ref{main})]
Since $Z$ is birationally isomorphic to the $d$-dimensional affine space $\A^d$ with $d=\dim(Z)\ge 1$, the groups $\Bir(X\times Z)$ and $\Bir(X\times\A^d)$ are isomorphic. So, it suffices to check that $\Bir(X\times\A^d)$ is {\sl not} a Jordan group. If $d=1$ the result follows from Theorem \ref{main}. If $d>1$ then
$X\times\A^d=(X\times\A^1) \times \A^{d-1}$ and one may view $\Bir(X\times\A^1)$ as the certain subgroup of $\Bir(X\times\A^d)$ and again Theorem \ref{main} gives us the desired result.
\end{proof}

The paper is organized as follows. Section \ref{Bir1} deals with the certain subgroup $\Bir_1(X\times\A^1)$ of $\Bir(E\times\A^1)$ that is generated by translations on $X$ and multiplications of the global coordinate $t$ on $\A^1$ by nonzero rational functions on $X$. We assert that  $\Bir_1(X\times\A^1)$  is not a Jordan group; obviously, this assertion implies that $\Bir(X\times\A^1)$  is also  not a Jordan group. In Section \ref{group} we discuss a {\sl symplectic geometry} related to certain analogues of Heisenberg groups that were introduced by Mumford \cite[Sect. 1]{Mumford66}. In Sect. \ref{theta}, using results of  Mumford \cite[Sect. 1]{Mumford66},  we realize these analogues as subgroups of  $\Bir_1(X\times\A^1)$, which allows us to prove that
$\Bir_1(A\times\A^1)$  is not a Jordan group.

I am grateful to Volodya Popov for a stimulating question, his interest in this paper and useful discussions.

\section{Birational automorphisms of products of an abelian variety and the affine line}
\label{Bir1}

Let $X$ be an abelian variety of positive dimension over $k$. If $y\in X(k)$ then we write $T_y$ for the translation map
$$T_y: X\to X, \ x \mapsto x+y.$$
As usual, we write $\div(f)$ for the divisor of a rational function $f \in k(X)^{*}$. Clearly,  $T_y^{*}f$ is the
rational function $x \mapsto f(x+y)$, whose divisor coincides with $T_y^{*}(\div(f))$. Let $t$ be the global coordinate on $\A^1$.

We write $\Bir_1(X\times\A^1)\subset \Bir(X\times\A^1)$
for the set of birational automorphisms of the form
$$A(y,f): X \times \A^1 \dashrightarrow  X \times \A^1, (x,t)\mapsto (x+y, f(x)\cdot t)=(T_y(x), f(x)\cdot t)$$
where $y$ runs through $X(k)$ and $f$ through $k(X)^{*}$. Actually, $\Bir_1(X\times\A^1)$ is a subgroup of  $\Bir(X\times\A^1)$.
Indeed, one may easily check that
$$A(y_2,f_2) A(y_1,f_1)=A(y_1+y_2, T_{y_1}^{*}(f_2) \cdot f_1) \in \Bir_1(X\times\A^1) $$
and the inverse of $A(y,f)$ in  $\Bir(X\times\A^1)$ coincides with $A(-y, T_{-y}^{*}(1/f)) \in \Bir_1(X\times\A^1)$.

Now Theorem \ref{main} becomes an immediate corollary of the following statement.

\begin{thm}
\label{main1} Let $X$ be an abelian variety of positive dimension  over $k$. Then  $\Bir_1(X\times\A^1)$  is not a Jordan group.
\end{thm}

We prove Theorem \ref{main1} in Section \ref{theta}.

\section{Group theory}
\label{group}
 Let ${\K}$ be a finite commutative group.
 Let $\hat{\K}:=\Hom(\K,k^{*})$ be the group of
{\sl characters} of $\K$. We write the group law on $\K$ additively and on $\hat{\K}$ multiplicatively. In particular,
we write $\l1$ for the trivial character of $\K$. Clearly, the groups ${\K}$ and $\hat{\K}$ are isomorphic
(noncanonically); in particular, they have the same order, which we denote by $N=N_{\K}$.

Let $\mu_N\subset k^{*}$ be the (sub)group of $N$th roots of unity.
Clearly, for every nonzero
$x\in \K$ there exists $\ell\in \hat{\K}$ with $\ell(x)\ne 1$. On the other hand,
$$N x=0, \ \ell(x)\in \mu_N \ \forall x\in \K, \ell \in \hat{\K}.$$
Let us consider the commutative finite group $\H_{\K}= \K\times \hat{\K}$ and the nondegenerate alternating
bi-additive form
$$e_{\K}: \H_{\K}\times \H_{\K} \to k^{*}, \ ((x,\ell),(x^{\prime},\ell^{\prime})) \mapsto
\ell^{\prime}(x)/\ell(x^{\prime}).$$
Clearly, all the values of $e_{\K}$ lie in $\mu_N$.

Let $E$ be an {\sl isotropic} subgroup of $\H_{\K}$ with respect to $e_{\K}$.
Let $E^{\bot}$ be the orthogonal complement of $E$ in $\H_{\K}$ with respect to $e_{\K}$. Then $E\subset E^{\bot}$ and the
 nondegeneracy of
$e_{\K}$  gives rise to a group isomorphism
$$\H_{\K}/E^{\bot}\cong \Hom(E,k^{*})=\hat{E}.$$
In particular, $E$ and $\H_{\K}/E^{\bot}$ have the same order.
 The inclusions $E \subset E^{\bot}\subset \H_{\K}$ imply that
$$\#(E)^2=\#(E)\cdot \#(\H_{\K}/E^{\bot})$$
divides $\#(\H_{\K})=N^2$ and therefore $\#(E)$ divides $N$. Since
$$N^2=\#(\H_{\K})=\#(E)\cdot \#(\H_{\K}/E),$$
the index of $E$ in $\H_{\K}$ is divisible by $N$.
This means that {\sl the index of every isotropic subgroup in $\H_{\K}$ is divisible by $N$ and therefore is greater or equal than $N$}.

 Following \cite[Sect. 1]{Mumford66}, let us consider the set
$$\G_{\K}=k^{*}\times \H_{\K}= k^{*}\times \K\times \hat{\K}$$ and introduce
on it the group structure, by defining the product
$$(a,x,\ell)\ (a^{\prime},x^{\prime},\ell^{\prime}):=(a a^{\prime}\ell^{\prime}(x),x+x^{\prime},\ell \ell^{\prime}).$$
One may naturally identify $k^{*}$ with the central subgroup $\{(a,0,\l1) \mid a \in k^{*}\}$. In fact, $\G_{\K}$ sits in
the short exact sequence
$$0 \to k^{*} \to \G_{\K}\stackrel{\pi}{\to} \H_{\K}\to 0$$
where $\pi: \G_{\K} \twoheadrightarrow \H_{\K}$ sends $(a,x,\ell)$ to $(x,\ell)$. One may easily check that if $g,
g^{\prime} \in \G_{\K}$ then
$$g g^{\prime} g^{-1} {g^{\prime}}^{-1}=e_{\K}(\pi(g),\pi(g^{\prime}))\in k^{*} \subset \G_{\K}.$$
It follows that a subgroup $\tilde{E}\subset \G_{\K}$ is {\sl commutative} if and only if its image $\pi(\tilde{E})$ is an
{\sl isotropic} subgroup in $\H_{\K}$; if this is the case then the index of  $\pi(\tilde{E})$ in $\H_{\K}$ is greater or equal than $N=N_{\K}$.

Clearly, the subset
$$\G_{\K}^{1}=\mu_N\times \H_{\K}= \mu_N\times \K\times \hat{\K}\subset \G_{\K}$$
is actually a subgroup of $\G_{\K}$. We have
 $\pi(\G_{\K}^{1})=\H_{\K}$. Therefore if $\tilde{E}$ is a commutative subgroup in $\G_{\K}^{1}$ then
the index of $\pi(\tilde{E})$ in $\H_K=\pi(\G_{\K}^{1})$ is greater or equal than $N=N_{\K}$. This implies that index of
 $\tilde{E}$ in $\G_{\K}^{1}$ is also greater or equal than $N=N_{\K}$.

\section{Mumford's theta groups}
\label{theta}

We keep all the notation and assumptions of Section \ref{Bir1}.



We denote by $\M_X$ the  constant sheaf (of rational functions) on the abelian variety  $X$ with respect to Zariski topology,
which assigns to every non-empty open subset $U$ of $X$ its field of rational functions $k(U)=k(X)$. For every $f \in
k(X)^{*}$ let us consider the sheaf (auto)morphism
$$[f]: \M_X \to \M_X$$
that is induced by multiplication by $f$ in $k(X)$. If $y \in X(k)$ then $T_y^{*}\M_X=\M_X$ and the induced (by functoriality)  sheaf
(auto)morphism $[f]: T_y^{*}[f]: T_y^{*}\M_X \to T_y^{*}\M_X$ coincides (after the identification of $T_y^{*}\M_X$ and $\M_X$) with
$$[T_y^{*}f]: \M_X \to \M_X.$$
 If $D$ is a divisor on $X$ then we view the invertible sheaf $\OC_X(D)$ as a certain subsheaf of $\M_X$ (see
\cite[Vol. II, Ch. 6, Sect. 1]{Sh}). Notice that for all $y\in X(k)$
$$T_y^{*}\OC_X(D)=\OC_X(T_y^{*}D).$$
If $D_1$ and $D_2$ are  linearly equivalent divisors on $X$  then  isomorphisms of invertible sheaves
$\OC_X(D_1)\cong \OC_X(D_2)$ are exactly (all the) morphisms of the form
$$[f]:\OC_X(D_1)\cong\OC_X(D_2)$$
with $\div(f)=D_1-D_2$. In particular, this set of isomorphisms is a $k^{*}$-torsor, since $\div(f)$ determines the
rational function $f$ up to multiplication by a nonzero constant.

If $[f]:\OC_X(D_1)\cong\OC_X(D_2)$ is an isomorphism of invertible sheaves and $y\in X(k)$ then the induced (by
functoriality) the isomorphism of invertible sheaves $T_y^{*}[f]:T_y^{*}\OC_X(D_1)\cong T_y^{*}\OC_X(D_2)$ coincides with
$$[T_y^{*}f]:\OC_X(T_y^{*}D_1)\cong \OC_X(T_y^{*}D_2).$$

Now let us choose an ample divisor on $X$ (e.g., a hyperplane section) and put $L=\OC_X(D)$. Then $L$ is an ample
invertible sheaf. Let us consider the (finite) commutative group
$$H(L)=\{x \in X(k)\mid L \cong T_x^{*}L \}.$$

\begin{rem}
\label{bign} Let $n$ be a positive integer. Then $nD$ remains ample, $\OC_X(nD)=L^n$ and
$$H(L^n)=\{x \in X(k)\mid nx \in H(L)\}$$
(see \cite[Sect. 1, Prop. 4]{Mumford66}). In particular, $H(L^n)$ contains the group $X_n$ of all points of order $n$ on
$X$. Since the order of $X_n$ is $n^{2\dim(X)}$ \cite[Ch. 2, Sect. 6]{MumfordAV}, the order of $H(L^n)$ is divisible by  $n^{2\dim(X)}$.

\end{rem}

 Following Mumford \cite[Sect. 1]{Mumford66}, let us consider the {\sl theta group} $\G(L)$ that consists of all pairs $(x,\phi)$ where $x \in H(L)$ and $\phi$ is an isomorphism
of invertible sheaves $L \cong T_x^{*}L$. The group law on $\G(L)$ is defined as follows. If $(x, \phi:L \cong
T_x^{*}L)\in \G(L)$ and $(y, \psi:L \cong T_y^{*}L)\in \G(L)$ then its composition $(y,\psi) (x,\phi)$ is defined as $$(x+y, T_{x}^{*}\phi \
\psi:L \cong T_y^{*}L \cong T_x^{*}(T_y^{*}L)=T_{x+y}^{*}L).$$ Taking into account our considerations in the beginning
of this Section and the equality $L=\OC_X(D)$, we conclude that $H(L)$ coincides with the set of $x\in X(k)$ such that
$D$ is linearly equivalent to $T_x^{*}D$, the theta group $\G(L)$ is the set of all pairs $(x,[f])$ where $x \in H(L)$
and $f$ is a nonzero rational function on $X$ such that
 $\div(f)=D-T_x^{*}D$.
  In addition, if $(y,[h]) \in \G(L)$ then
$$(x,[f]) (y, [h])=(x+y, [T_x^{*}h \cdot f]) \in \G(L).$$

\begin{rem}
\label{bigindex} It is known \cite[Sect. 1, Cor. of Th. 1]{Mumford66} that there exists a finite sequence of positive integers (elementary divisors) $\delta=(d_1, \dots , d_r)$ such that $d_{i+1}\mid d_i$ and the finite commutative group
$K(\delta)=\oplus_{i=1}^r \Z/d_i\Z$ enjoys the following properties:

\begin{itemize}
\item
 $H(L)$ is isomorphic to $\H_{K(\delta)}$;
 \item
  the groups $\G_{K(\delta)}$ and $\G(L)$ are
isomorphic.
\end{itemize}

 Applying the results of Section
\ref{group}, we conclude that $\G(L))$ contains a finite subgroup $G$ that enjoys the following property: every
commutative subgroup in  $G$ has index that is greater or equal than $\#(K(\delta))=\sqrt{\#(H(L))}$.
\end{rem}

\begin{proof}[Proof of Theorem \ref{main1}]
Comparing the multiplication formulas for $(x,[f])$'s and $A(y,f)$'s (Sect. \ref{Bir1}), we conclude that the embedding
$$\G(L) \hookrightarrow \Bir_1(X\times \AA^1), \ (y,[h])\mapsto A(y,h)$$
is actually a group homomorphism. So $\G(L)$
is isomorphic to a subgroup of $\Bir_1(X\times \AA^1)$. Applying this assertion to all ample divisors  $nD$ and invertible sheaves $L^n=\OC_X(nD)$ (where $n$ is a
positive integer) and combining it with Remarks \ref{bign}  and \ref{bigindex}, we conclude that for every positive
integer $n$ there exists a finite subgroup
$$G \subset \G(L^n)\hookrightarrow \Bir_1(X\times \AA^1)$$
that enjoys the following property: every commutative subgroup in $G$ has index that is greater or equal than
${(n^{2\dim(X)})}^{1/2}=n^{\dim(X)}$; in particular, this index is greater or equal than $n$. This proves that
$\Bir_1(X\times \AA^1)$ is {\sl not} a Jordan group.
\end{proof}

\end{document}